\documentclass[12pt]{amsart}

\usepackage{xcolor}

\usepackage[spanish,english]{babel}
\usepackage[latin1]{inputenc}
\usepackage{latexsym}

\usepackage{graphicx}
\usepackage{booktabs,colortbl}

\usepackage[mathscr]{eucal}

\usepackage{amsmath, amsthm, amscd, amsfonts, amssymb, graphicx, color, xypic}
\usepackage[all]{xy}

\newtheorem{theorem}{Theorem}

\newtheorem{corollary}[theorem]{Corollary}
\newtheorem{example}[theorem]{Example}

\newtheorem{remark}[theorem]{Remark}

\newtheorem{definition}[theorem]{Definition}

\begin{document}

		\title[McShane-Whitney extensions for fuzzy Lipschitz maps]{A revised and extended version of McShane-Whitney extensions for fuzzy Lipschitz maps}

	
	\author[E. Jim\'enez-Fern\'andez et al.]{Eduardo Jim\'enez-Fern\'andez}
	\address[E. Jim\'enez-Fern\'andez]{Departamento de Teor\'{\i}a e Historia Econ\'omica, Campus Universitario de Cartuja, Universidad de Granada, 18071 Granada, Spain}
	\email{edjimfer@ugr.es}
	
	\author[]{Jes\'us Rodr\'{\i}guez-L\'opez}
	\address[J. Rodr\'{\i}guez-L\'opez]{Instituto Universitario de Matem\'atica Pura y Aplicada, Universitat Polit\`ecnica de Val\`encia, Camino de Vera s/n, 46022 Valencia, Spain}\thanks{Last author's research
		is part of the project PID2022-139248NB-I00 funded by MICIU/AEI/10.13039/501100011033 and ERDF/EU}
	\email{jrlopez@mat.upv.es}
	
	\author[]{Aurora S\'anchez-Mart\'{\i}n-Orozco}
	\address[A. S\'anchez-Mart\'{\i}n-Orozco]{Instituto Universitario de Matem\'atica Pura y Aplicada, Universitat Polit\`ecnica de Val\`encia, Camino de Vera s/n, 46022 Valencia, Spain}
	\email{aurora.sanchez@uv.es}
	
	\author[]{Enrique A. S\'anchez-P\'erez}
	\address[E. A. S\'anchez-P\'erez]{Instituto Universitario de Matem\'atica Pura y Aplicada, Universitat Polit\`ecnica de Val\`encia, Camino de Vera s/n, 46022 Valencia, Spain}
	\email{easancpe@mat.upv.es}

	\subjclass[2010]{26A16; 54E70; 54C20.}

		\begin{abstract}
			In the paper [E. Jim\'enez-Fern\'andez, J. Rodr\'{\i}guez-L\'opez, E. A. S\'anchez-P\'erez, Fuzzy Sets and Systems 406 (2021),66-81], a McShane-Whitney extension theorem is presented for real-valued fuzzy Lipschitz maps between fuzzy metric spaces. Specifically, the codomain space is considered as a so-called Euclidean fuzzy metric space $(\mathbb{R},M_{\phi,g},\ast).$ However, while the function $\phi$ is only required to be increasing, some results of the paper implicitly assume that $\phi$ is invertible, even though this is not explicitly stated. We propose here an alternative possibility that only requires $\phi$ to be also left-continuous. 
		\end{abstract}
		
		\keywords{fuzzy metric; fuzzy Lipschitz map; extension; McShane-Whitney Theorem.}

	\date{}
	\maketitle

	The concept of a Euclidean fuzzy metric space was introduced in \cite{JFRLSP21}. Here, we present a significant subcategory of these fuzzy metric spaces that will be crucial for our subsequent discussion.

	\begin{definition}[\cite{JFRLSP21}]
		A fuzzy metric $(M,\ast)$ on $\mathbb{R}$ is said to be a \emph{Euclidean fuzzy metric} if there are  functions  $\phi:[0,+\infty) \to [0,+\infty), g:(0,+\infty)\to (0,+\infty)$ such that
		$\phi$ is increasing, and
		$$
		M(x,y,t)=1- \phi(|x-y|) g(t)
		$$
		for all $x,y\in \mathbb{R}$, $t>0.$
		In this case we will denote $M$ by $M_{\phi,g}$  and we will say that $(\mathbb{R},M_{\phi,g},\ast)$ is a Euclidean fuzzy metric space.
		
		If $\phi$ is also defined at $+\infty$ $(\phi(+\infty)\in [0,+\infty])$ and $\phi$ is also left-continuous, then we say that $(\mathbb{R},M_{\phi,g},\ast)$ is a left-continuous Euclidean fuzzy metric space.
	\end{definition}
	
	\begin{remark}\label{rem1}
		Notice that if $(M_{\phi,g},\ast)$ is a Euclidean fuzzy metric on $\mathbb{R},$ then
		\begin{itemize}
			\item[(1)] $\phi^{-1}(0)=\{0\};$
			\item[(2)] $\phi$ is bounded{\rm ;}
			\item[(3)] $g$ is bounded by $\frac{1}{\sup\{\phi(t):t>0\}}.$
		\end{itemize}  Proof. 
		
		\noindent (1) Given $x\in \mathbb{R},$ then $M_{\phi,g}(x,x,t)=1-\phi(0)g(t)=1$ for all $t>0.$ If $\phi(0)\neq 0,$ then $g(t)=0$ for all $t>0$ so $M_{\phi,g}(x,y,t)=1$ for all $x,y\in \mathbb{R}.$ Hence $(M_{\phi,g},\ast)$ is not a fuzzy metric, a contradiction. Therefore $\phi(0)=0.$ 
		
		Moreover, $\phi(t)\neq 0$ for every $t>0.$  Otherwise, there exists $s>0$ such that $\phi(s)=0$ so $M_{\phi,g}(s,0,t)=1-\phi(|s-0|)g(t)=1$ for all $t>0,$ which is a contradiction. 
		
		\noindent (2) If $\phi$ is not bounded, given $t>0,$ we can find $s>0$ such that $\phi(s)>\frac{1}{g(t)}$. Hence
		$$M_{\phi,g}(s,0,t)=1-\phi(s)g(t)<0$$
		which is not possible since $M_{\phi,g}$ is a fuzzy set.	
		
		\noindent (3)  If $g$ is not bounded by $\frac{1}{\sup\{\phi(t):t>0\}},$ we can find $t',s>0$ such that $g(s)>\frac{1}{\phi(t')}$ so
		$$M_{\phi,g}(t',0,s)=1-\phi(t')g(s)<0,$$
		and this is not possible.

	\end{remark}

	\begin{remark}
		Notice that given $\phi:[0,+\infty]\to [0,+\infty],$ we can define $\phi^*:[0,+\infty]\to [0,+\infty]$ as
		$$\phi^*(y)=\sup\{x\in [0,+\infty]:\phi(x)\leq y\},$$
		for all $y\in [0,+\infty]$ $(\sup\varnothing =0).$ Then it is easy to check that 
		$$\mathrm{id}\leq \phi^*\circ \phi,$$
		where $\mathrm{id}$ is the identity function on $[0,+\infty].$ Moreover, if $\phi$ is increasing, left-continuous and $\phi(0)=0,$ then
		$$\phi\circ\phi^\ast\leq \mathrm{id}.$$
		For proving this, let $y\in [0,+\infty].$ Then
		$$\phi(\phi^\ast (y))=\phi\big(\sup \{x\in [0,+\infty]:\phi(x)\leq y\}\big)=\sup\{\phi(x): \phi(x)\leq y\}\leq y,$$
		where, in the second equality, we have used that $\phi$ is left-continuous and increasing.
		
		We observe that the existence of $\phi^*$ fulfilling the two previous inequalities is equivalent to stating that $\phi$ possesses a right-adjoint, that is, $\phi$ is join-preserving \cite{BookQuantales}.
	\end{remark}

	\begin{remark}
		If $(M_{\phi,g},\ast)$ is a left-continuous Euclidean fuzzy metric on $\mathbb{R}$ with respect to $\phi,g,$ then $\phi([0,+\infty))\subseteq [0,+\infty),$ so $(M_{\hat{\phi},g},\ast)$ is a Euclidean fuzzy metric, where $\hat{\phi}$ denotes the restriction to $[0,+\infty)$ of $\phi$.
		
		Furthermore, given an increasing function $\phi:[0,+\infty)\to [0,+\infty),$ we can construct $\bar{\phi}:[0,+\infty]\to [0,+\infty]$ as
		$$\bar{\phi}(x)=\sup \{\phi(y):y<x\},$$
		for all $x\in [0,+\infty].$ Then $\bar{\phi}$ is increasing and left-continuous.
	\end{remark}

	In Section 5 of \cite{JFRLSP21}, the authors study the problem of extending fuzzy Lipschitz functions defined from a fuzzy metric space $(X,M,\ast)$ to a Euclidean fuzzy metric space $(\mathbb{R},N_{\phi,g},\circledast).$ Nevertheless, their proofs require the existence of the inverse of $\phi$ on $[0,+\infty).$ This requirement cannot be assumed for an arbitrary Euclidean fuzzy metric space because $\phi$ is only an increasing function, and as noted in Remark \ref{rem1}, $\phi$ is bounded.

	Nevertheless, if we consider $(\mathbb{R},N_{\phi,g},\circledast)$ as a left-continuous Euclidean fuzzy metric space, we can always guarantee the existence of the right-adjoint $\phi^*.$ This right-adjoint appears to be a suitable candidate for replacing the inverse $\phi^{-1}$ in the proofs presented in \cite{JFRLSP21}. Nonetheless, this does not fully resolve the issue, as we will explain next.
	
	If $(\mathbb{R},M_{\phi,g},\ast)$ is a left-continuous Euclidean fuzzy metric space, then $\phi$ is bounded. If $K=\sup\{\phi(t):t\in [0,+\infty]\}=\phi(+\infty),$ since $\phi$ is left-continuous and increasing, then 
	$$\phi^*(s)=\sup\{t\in [0,+\infty]:\phi(t)\leq s\}=\sup \{t\in [0,+\infty]\}=+\infty,$$
	for every $s\geq K.$ Hence, $\phi^*$ always takes the value $+\infty.$ This is an unpleasant fact that causes that we also need an extra boundedness assumption. 
	In this way, we can reformulate \cite[Theorem 21]{JFRLSP21} as follows:

	\begin{theorem} \label{mainth}
		Let $(X,M,\ast)$ be a fuzzy metric space and $(\mathbb{R},N_{\phi,g},\circledast)$ be a left-continuous Euclidean fuzzy metric space. Let $S\subseteq X$ and suppose that  $f:(S,M,\ast) \times (0,+\infty) \to  (\mathbb R, N_{\phi,g},  \circledast)$
		is a fuzzy Lipschitz map, that is, for each $t>0$ there exists $K(t)>0$ such that 
		$$1 - N_{\phi,g}( f (x,t), f (y,t), t)\leq K(t)(1 - M(x, y, t)),$$
		for every $x,y\in S.$
		Suppose that 
		$$\frac{K(t)}{g(t)} \, \big( 1- M(x,y,t) \big)< \phi(+\infty),$$
		for all $x,y\in X$, $t>0.$ 	Then the function $f$ can be extended as a fuzzy Lipschitz map to $X \times (0,+\infty).$
	\end{theorem}
	
	\begin{proof}
		Since $f:(S,M,\ast) \times (0,+\infty) \to  (\mathbb R, N_{\phi,g}, \circledast)
		$ is a fuzzy Lipschitz map, then for each $t>0$
		there is a positive real number $K(t)$ such that
		\begin{align*}
			1- N_{\phi,g}(f(x,t),f(y,t),t) & \le K(t) \, \Big( 1- M(x,y,t) \Big),\\
			\phi(|f(x,t)-f(y,t)|) &\le \frac{K(t)}{g(t)} \, \Big( 1- M(x,y,t) \Big)
		\end{align*}
		for all $x,y\in S.$  Since  $\mathrm{id}\leq \phi^*\circ \phi$ and $\phi^*$ is increasing, then
		\begin{equation}
			|f(x,t)-f(y,t)| \le (\phi^*\circ\phi)(|f(x,t)-f(y,t)| )\le \phi^{*} \left( \frac{K(t)}{g(t)} \, \big( 1- M(x,y,t) \big) \right)\!,	\label{eq:Lip}
		\end{equation}
		for all $x, y \in S.$ 
		
		For each $t>0,$ define $d_t:X\times X\to [0,+\infty)$ as
		
		$$
		d_t(x,y)=\inf_{n\in\mathbb{N}} \left\{ \sum_{i=1}^n \phi^{*}\left( \frac{K(t)}{g(t)} (1 - M(x_i,x_{i+1},t)) \right) \;\middle|\; x_1{=}x,\; x_{n+1}{=}y,\; x_i \in X \right\}\!,
		$$
		for all $x,y\in X$. Since $\frac{K(t)}{g(t)} \, \big( 1- M(x,y,t) \big)< \phi(+\infty)$
		for all $x,y\in X,$ then $d_t$ is finite. 
		
		It is straightforward to show that $d_t$ is a pseudometric on $X.$
		
		Due to the fact that $|\cdot|$ is a norm on $\mathbb{R}$, we automatically get from the inequality (\ref{eq:Lip}) above that for each $t>0$
		$$|f(x, t)-f(y, t)|\leq d_t(x,y)$$
		for all $x,y\in S.$ Consequently, $f(\cdot,t):(X,d_t)\to (\mathbb{R},|\cdot|)$ is a Lipschitz function. So we can apply the classical McShane-Whitney extension theorem to $f(\cdot,t)$ for all $t>0,$ obtaining a Lipschitz extension $\hat f(\cdot, t)$. Then 
		\[
		|\hat f(x,t)- \hat f(y,t)| \le d_{t}(x,y)\leq \phi^{*} \left(\frac{K(t)}{g(t)} \, \big( 1- M(x,y,t) \big) \right)\!,
		\]
		for all $x,y\in X$ and since $\phi \circ \phi^* \leq \mathrm{id},$ we have that
		
		\[
		\phi(|\hat f(x,t)- \hat f(y,t)|) \le (\phi\circ \phi^{*}) \left(\frac{K(t)}{g(t)} \, \big( 1- M(x,y,t) \big) \right)\le \frac{K(t)}{g(t)} \, \big( 1- M(x,y,t)\big).
		\]
		
		Therefore
		\[
		1- N_{\phi,g} (\hat f(x,t), \hat f(y,t),t) = \phi(|\hat f(x,t)-\hat f(y,t)|) g(t) \le K(t) \, \big( 1- M(x,y,t) \big),
		\]
		for all $x,y\in X.$ Since $t>0$ is arbitrary, then $\hat f:X\times (0,+\infty)\to Y$ is fuzzy Lipschitz.
	\end{proof}

	Moreover, \cite[Corollay 22]{JFRLSP21} would be stated as follows.
	
	\begin{corollary}  \label{corex}
		Let $(X,M,\ast)$ be a fuzzy metric space and  $(\mathbb R, N_{\phi,g}, \circledast)$ be a left-continuous Euclidean fuzzy metric space. Let $S \subseteq X$ and suppose that $f:(S,M,\ast) \times (0,+\infty) \to (\mathbb R,  N_{\phi,g}, \circledast)$
		is a fuzzy Lipschitz map  with extended dilation $K(t).$ Assume also that for every $t>0$ the map  $\rho_t:X\times X\to [0,+\infty)$ given by
		$$\rho_{t}(x,y) :=\phi^{*} \left(\frac{K(t)}{g(t)} \, \big( 1- M(x,y,t) \big) \right)\!,  \quad x,y \in X,
		$$
		is a metric on $X$.
		Then the function $f$ can be extended as a fuzzy Lipschitz map to $X \times (0,+\infty).$ 
		
		Moreover, two suitable extensions of $f$ are
		$$f^M(x,t)=\sup_{s\in S}\left\{f(s,t)-\phi^{*}\left(\frac{K(t)}{g(t)}(1-M(x,s,t))\right)\right\}\!,$$
		$$f^W(x,t)=\inf_{s\in S}\left\{f(s,t)+\phi^{*}\left(\frac{K(t)}{g(t)}(1-M(x,s,t))\right)\right\}\!,$$
		for all $x\in X$ and all $t>0.$
		
	\end{corollary}

	In \cite[Section 6]{JFRLSP21}, the authors present McShane-Whitney extension formulas through two specific examples. However, these formulas depend on the existence of the inverse of the function  $\phi: [0,+\infty)\to [0,+\infty)$ given by $\phi(x)=\min\{x,1\},$ for all $x\in [0,+\infty).$ But, of course, $\phi$ is not bijective. 
	
	We next provide correct formulas for those examples based on our previous discussion.

	\begin{example}
		Take a metric space $(X,d)$ and  construct the fuzzy metric space  $(X,M_k,\ast_{\text{\L}})$  defined by
		$$
		M_k(x,y,t)= 1- \frac{\min\{d(x,y),k\}}{h(t)}, \quad x,y\in X,
		$$
		where $\ast_{\text{\L}}$ is the \L ukasiewicz t-norm, $k>0$ and $h:(0,+\infty)\to (k,+\infty)$ is an increasing continuous function (see \cite[Example 6]{GregMoriSape10}).
		
		Consider the left-continuous Euclidean fuzzy metric space $(\mathbb{R},N_E,\ast_{\text{\L}})$, given by
		$$
		N_E(x,y,t)= 1-\frac{\min\left\{|x-y|,1\right\}}{2}, \quad x,y \in \mathbb{R}.
		$$
		Notice that, in this case, $\phi(x)=\frac{\min\left\{x,1\right\}}{2}$ so
		
		$$\phi^*(x)=\begin{cases}2x&\text{ if }0\leq x<\frac{1}{2},\\
			+\infty&\text{ if }x\geq \frac{1}{2},\end{cases}$$ for all $x\in [0,+\infty].$

		Let $S \subseteq X.$
		Suppose that the  function $f:(S,M_k,\ast_{\text{\L}})\times (0,+\infty) \to (\mathbb{R},N_E, \ast_{\text{\L}})$ is a fuzzy contractive map verifying that 
		$$1- N_E(f(x,t),f(y,t),t) \le K(t) \big( 1- M_k(x,y,t) \big), \quad x, y \in S, t>0,
		$$
		where $K(t)<\frac{1}{2}$, which can be rewritten as
		$$\phi(|f(x,t)-f(y,t)|)=\frac{\min\{|f(x,t)-f(y,t)|,1\}}{2}\leq \frac{K(t)}{h(t)}\min\{d(x,y),k\}.$$
		So there is a function $Q:\mathbb{R}^+\to\mathbb{R}^+$ such that
		$$\phi(|f(x,t)-f(y,t)|)\leq Q(t)\min\{d(x,y),k\},$$
		for all $x,y\in S$ and all $t>0.$ Since $\phi^*$ is increasing, then
		\begin{align*}
			|f(x,t)-f(y,t)|&\leq \phi^*(\phi(|f(x,t)-f(y,t)|) )\leq \phi^*(Q(t)\min\{d(x,y),k\})\\
			&=2Q(t)\min\{d(x,y),k\},
		\end{align*}
		for all $x,y\in S$ and all $t>0$ since $Q(t)\min\{d(x,y),k\}<\frac{1}{2}.$

		Then the McShane and Whitney extensions of $f$ to $X\times (0,+\infty)$ are given by
		$$
		{{f^M}}(x,t):=\sup_{s\in S}\big\{f(s,t)-2Q(t) \,\min\{d(s,x),k\}\big\}, \quad x \in X,
		$$
		and
		$$
		f^W(x,t):=\inf_{s\in S}\big\{f(s,t)+2Q(t) \,\min\{d(s,x),k\}\big\}, \quad x \in X.
		$$

		Thus, a possible family of extensions would be given by functions as
		\begin{align*}
			f_{\alpha(t)}(x,t) =& \alpha(t) \, {{f^M}}(x,t) + (1-\alpha(t) )\, {{f^W}}(x,t)
			\\
			=&
			\alpha(t) \, \sup_{s\in S}\big\{f(s,t)-2Q(t) \,\min\{d(s,x),k\}\big\}  \\&+ (1-\alpha(t) )\, \inf_{s\in S}\big\{f(s,t)+2Q(t) \,\min\{d(s,x),k\}\big\},
		\end{align*}
		for $\alpha: [0,+\infty) \to [0,1]$,	$x \in X,$ $t>0.$

	\end{example}	
	
	\color{black}
	
	\begin{example}
		Consider the stationary fuzzy metric $(M,\cdot)$ given by 
		$$
		M(x,y,t)= e^{-d(x,y)}, \quad x, y \in X,
		$$
		where $(X,d)$ is a metric space.
		As above, let $S\subseteq X$ and $f:(S,M,\cdot)\to (\mathbb{R},N_E,\ast_{\text{\L}})$ be a fuzzy Lipschitz function such that for each $t>0$ we can find $0<K(t)<\frac{1}{2}$ such that
		$$
		1- N_E(f(x),f(y),t) \le K(t) \big( 1- M(x,y,t) \big), \quad x, y \in S,
		$$
		for all $x,y\in X,$  which can be rewritten as
		$$\frac{\min\{|f(x,t)-f(y,t)|,1\}}{2}\leq K(t)\left(1-e^{-d(x,y)}\right)\!.$$
		Notice that since $\cdot\geq \ast_{\text{\L}},$ we have by \cite[Proposition 5]{GregMoriSape10} that $1-e^{-d(x,y)}$ is a metric on $X$.
		
		As in the previous example, for functions $\alpha: [0,+\infty) \to [0,1]$, we obtain that a
		family of fuzzy Lipschitz extensions of $f$ is
		$$
		f_{\alpha(t)}(x,t) = \alpha(t) \, {{f^M}}(x,t) + \big(1-\alpha(t) \big)\, {{f^W}}(x,t), \quad x \in X, \,\, t>0,
		$$
		where
		$$
		{{f^M}}(x,t)=\sup_{s\in S} \Big\{f(s,t)-2K(t) \, \Big( 1- e^{- d(s,x)} \Big) \Big\}, \quad x \in X, \,\, t>0,
		$$
		and
		$$
		f^W(x,t)=\inf_{s\in S} \Big\{f(s,t)+2K(t) \, \Big( 1- e^{- d(s,x)} \Big) \Big\},\quad x \in X, \,\,  t>0.
		$$
	\end{example}


\begin{thebibliography}{99}
		
		\bibitem{BookQuantales}
		
		
		
		P.~Eklund, J.~Guti\'errez-Garc\'{\i}a, U.~H\"ohle, and J.~Kortelainen,
		\emph{Semigroups in complete lattices. {Q}uantales, modules and related
			topics}, Springer, 2018.
		
		
		\bibitem{GregMoriSape10}
		V. Gregori, S. Morillas,
		and A. Sapena, \textit{  On a class of completable fuzzy metric spaces,}  Fuzzy Sets and Systems 161 (2010), 2193--2205.
		
		\bibitem{JFRLSP21} E. Jim\'enez-Fern\'andez, J. Rodr\'{\i}guez-L\'opez, and E. A. S\'anchez-P\'erez, \emph{McShane-Whitney extensions for fuzzy Lipschitz maps}, Fuzzy Sets and Systems 406 (2021), 66-81.
	\end{thebibliography}
\end{document}